\documentclass[a4paper]{article}
\usepackage{amsmath,amscd}
\usepackage{amssymb}
\usepackage{amsthm}
\usepackage{amsbsy,amstext,amsxtra,amsopn}
\makeatletter
  
  \@addtoreset{equation}{section}
\makeatother
\usepackage{latexsym}
\begin{document}
\title{{\bf Geometrically simple quasi-abelian varieties}}
\author{Yukitaka Abe}
\date{ }
\maketitle

\noindent
{\bf Abstract}\\
We define the geometric simpleness for toroidal groups.
We give an example of quasi-abelian variety which is geometrically
simple, but not simple.
We show that any quasi-abelian variety is isogenous to a product
of geometrically simple quasi-abelian varieties.
We also show that the ${\mathbb Q}$-extension of the ring of all
endomorphisms of a geometrically simple quasi-abelian variety is a
division algebra over ${\mathbb Q}$.

\ 
\footnote{
{\bf keywords:} Geometrically simple quasi-abelian varieties,
Isogeny}\ 
\footnote{
{\bf Running head:} Simple quasi-abelian varieties}\ 

\newtheorem{definition}{Definition}

\newtheorem{lemma}[definition]{Lemma}

\newtheorem{theorem}[definition]{Theorem}

\newtheorem{proposition}[definition]{Proposition}

\newtheorem{corollary}[definition]{Corollary}

\section{Introduction}
Let ${\rm End}_{{\mathbb Q}}(X)$ be the ${\mathbb Q}$-extension
of the ring of all endomorphisms of a toroidal group $X$.
To study ${\rm End}_{{\mathbb Q}}(X)$, we defined the
simpleness of toroidal groups as follows (Definition 2.1 in
\cite{abe}):\\
A toroidal group is simple if it does not contain a toroidal
subgroup apart from itself and zero.\\
We note that toroidal subgroups are not always closed in general.
It immediately follows from this definition that if $X$ is a
simple toroidal group, then ${\rm End}_{{\mathbb Q}}(X)$ is a
division algebra over ${\mathbb Q}$ (Lemma 3.3 in \cite{abe}).\par
Any abelian variety ${\mathbb A}$ is isogenous to a product
of simple abelian varieties.
Unfortunately, it does not hold that any quasi-abelian variety is
isogenous to a product of simple quasi-abelian varieties.
We think that the above definition of simpleness is too
strong. We define the geometric simpleness for toroidal
groups (Definition 1). We show that ${\rm End}_{{\mathbb Q}}(X)$ is a
division algebra over ${\mathbb Q}$ if $X$ is a geometrically simple
quasi-abelian varietyany, and that any quasi-abelian variety is
isogenous to a product of geometrically simple quasi-abelian
varieties.

\section{Geometric simpleness}
We consider a toroidal group $X = {\mathbb C}^n/\Gamma $ with
${\rm rank}\; \Gamma = n+m$. Let $\pi : {\mathbb C}^n 
\longrightarrow X$ be the projection. If $Y$ is a connected complex
Lie subgroup of $X$, then $\pi ^{-1}(Y) = E + \Gamma $,
where $E$ is a complex linear subspace of ${\mathbb C}^n$.
In this case, $Y$ is closed if and only if $E + \Gamma $ is
closed in ${\mathbb C}^n$.

\begin{definition}
A toroidal group is said to be geometrically simple if it
does not contain a closed toroidal subgroup apart from itself
and zero.
\end{definition}

\begin{lemma}
Let ${\mathbb A} = {\mathbb C}^2/\Lambda $ be a 2-dimensional
complex torus with the following period matrix
\begin{equation*}
P =
\begin{pmatrix}
1 & 0 & \sqrt{-1}r^3 & r\\
0 & 1 & r & \sqrt{-1}\\
\end{pmatrix}
,
\end{equation*}
where $r$ is a positive number such that 
$1, r, r^2, r^3 $ are linearly independent over ${\mathbb Q}$.
Then ${\mathbb A}$ is a simple abelian variety.
\end{lemma}

\begin{proof}
It is obvious that ${\mathbb A}$ is an abelian variety. We show
that it is simple.\par
Any $\lambda \in \Lambda $ has the unique representation
\begin{equation}
\lambda = 
\begin{pmatrix}
a_1 + \sqrt{-1}a_3 r^3 + a_4 r \\
a_2 + a_3 r + \sqrt{-1}a_4 \\
\end{pmatrix},\quad
a_1, a_2, a_3, a_4 \in {\mathbb Z}.
\end{equation}
Let $L$ be a complex line with
$L \cap \Lambda \not= \{ 0 \}$. Take
$\lambda ^{(0)} \in L \cap \Lambda $ such that
\begin{equation*}
\| \lambda ^{(0)}\| = \min \{ \| \lambda \| ;
\lambda \in L \cap \Lambda , \lambda \not= 0 \}.
\end{equation*}
It is written as
\begin{equation*}
\lambda ^{(0)}= 
\begin{pmatrix}
a_1^{(0)} + \sqrt{-1}a_3^{(0)} r^3 + a_4^{(0)} r \\
a_2^{(0)} + a_3^{(0)} r + \sqrt{-1}a_4^{(0)} \\
\end{pmatrix},\quad
a_1^{(0)}, a_2^{(0)}, a_3^{(0)}, a_4^{(0)} \in {\mathbb Z}.
\end{equation*}
We note that $(a_1^{(0)}, a_2^{(0)}, a_3^{(0)}, a_4^{(0)})=1$
and $L = {\mathbb C}\lambda ^{(0)}$.
Let $\lambda \in L \cap \Lambda $ with $ \lambda \not= 0$. Then
there exists $\zeta \in {\mathbb C}^{*}$ such that
$\lambda = \zeta \lambda ^{(0)}$. If we represent $\lambda $ as (2.1),
then we have
\begin{equation*}
\left\{ 
\begin{aligned}
a_1 + \sqrt{-1}a_3 r^3 + a_4 r & = \zeta (
a_1^{(0)} + \sqrt{-1}a_3^{(0)} r^3 + a_4^{(0)} r),\\
a_2 + a_3 r + \sqrt{-1}a_4 & = \zeta (
a_2^{(0)} + a_3^{(0)} r + \sqrt{-1}a_4^{(0)}).
\end{aligned}
\right.
\end{equation*}
Therefore we obtain
\begin{equation}
\begin{aligned}
(a_1 + \sqrt{-1}a_3 r^3 + a_4 r)(
a_2^{(0)} + a_3^{(0)} r + \sqrt{-1} a_4^{(0)})\quad \quad \quad \\
= (a_1^{(0)} + \sqrt{-1}a_3^{(0)} r^3 + a_4^{(0)} r)(
a_2 + a_3 r + \sqrt{-1} a_4).
\end{aligned}
\end{equation}
Since $1,r,r^2,r^3$ are linearly independent over ${\mathbb Q}$,
it follows from (2.2) that
\begin{equation}
\left\{
\begin{aligned}
a_1a_2^{(0)} &= a_1^{(0)}a_2,\\
a_1a_3^{(0)} + a_4a_2^{(0)} &= a_1^{(0)}a_3 +
a_4^{(0)}a_2,\\
a_4a_3^{(0)} &= a_4^{(0)}a_3,\\
a_1a_4^{(0)} &= a_1^{(0)}a_4,\\
a_3a_2^{(0)} &= a_3^{(0)}a_2.\\
\end{aligned}
\right.
\end{equation}
\par
We first consider the case $a_3^{(0)}a_4^{(0)} \not= 0$.
By the third equality of (2.3), if $a_4 = 0$, then $a_3 = 0$.
In this case we have $a_2 = 0$ by the fifth equality of
(2.3). Furthermore, $a_1 = 0$ by the second equality of
(2.3). Then $\lambda = 0$, which is a contradiction.
Hence we have $a_4 \not= 0$. Multiplying the third
equality of (2.3) by $a_1$ and using the fourth equality
of (2.3), we obtain
\begin{equation*}
a_1a_4a_3^{(0)} = a_1a_4^{(0)}a_3 = a_1^{(0)}a_4a_3.
\end{equation*}
Since $a_4 \not= 0$, we have $a_1a_3^{(0)} = a_1^{(0)}a_3$.
From the second equality of (2.3) it follows that
$a_4a_2^{(0)} = a_4^{(0)}a_2$. Then we obtain the
following equalities
\begin{equation}
\left\{
\begin{aligned}
a_1a_2^{(0)} &= a_1^{(0)}a_2,\\
a_1a_3^{(0)} &= a_1^{(0)}a_3,\\
a_1a_4^{(0)} &= a_1^{(0)}a_4,\\
a_2a_3^{(0)} &= a_2^{(0)}a_3,\\
a_2a_4^{(0)} &= a_2^{(0)}a_4,\\
a_3a_4^{(0)} &= a_3^{(0)}a_4.
\end{aligned}
\right.
\end{equation}
Let $p$ be a prime factor of $a_1^{(0)}$. Since
$(a_1^{(0)}, a_2^{(0)}, a_3^{(0)}, a_4^{(0)}) = 1$, there
exists $i$ such that $p$ does not divide $a_i^{(0)}$.
Now we consider integral solutions $(a_1, a_i)$ of
$a_1a_i^{(0)} = a_1^{(0)}a_i$. Let
$d_{1i} := (a_1^{(0)}, a_i^{(0)})$. Then we have
$a_1^{(0)} = b_1^{(0)}d_{1i}$ and $a_i^{(0)} = b_i^{(0)}d_{1i}$,
where $(b_1^{(0)}, b_i^{(0)}) = 1$.
The integral solutions of the above equation are written as
\begin{equation*}
a_1 = b_1^{(0)}k_{1i},\quad
a_i = b_i^{(0)}k_{1i}\quad
(k_{1i} \in {\mathbb Z}).
\end{equation*}
Since $p \nmid a_i^{(0)}$, we have $p \nmid d_{1i}$.
Then $p | b_1^{(0)}$. Therefore $p | a_1$. Hence we have
$a_1 = a_1^{(0)} k_1$, $k_1 \in {\mathbb Z}$, for $p$ is an
arbitrary prime factor of $a_1^{(0)}$. Similarly, we obtain
\begin{equation*}
a_2 = a_2^{(0)} k_2,\quad a_3 = a_3^{(0)} k_3,\quad
a_4 = a_4^{(0)} k_4,
\end{equation*}
where $k_2, k_3, k_4 \in {\mathbb Z}$. By (2.4) we obtain
$k_1 = k_2 = k_3 = k_4$. Thus we have $L \cap \Lambda = {\mathbb Z}\lambda ^{(0)}$.
\par
We next consider the case that $a_3^{(0)} \not= 0$ and
$a_4^{(0)} = 0$. From the third equality of (2.3) it follows that
$a_4 = 0$. In this case the equalities (2.3) become
\begin{equation*}
\left\{
\begin{aligned}
a_1a_2^{(0)} &= a_1^{(0)}a_2,\\
a_1a_3^{(0)} &= a_1^{(0)}a_3,\\
a_2a_3^{(0)} &= a_2^{(0)}a_3.\\
\end{aligned}
\right.
\end{equation*}
Therefore, we obtain the same conclusion
$L \cap \Lambda = {\mathbb Z}\lambda ^{(0)}$.\par
We also obtain $L \cap \Lambda = {\mathbb Z}\lambda ^{(0)}$ in other cases
by the same argument. Then ${\mathbb A}$ does not contain a
1-dimensional complex torus.
\end{proof}

\noindent
{\bf Example.}
Let
\begin{equation*}
P =
\begin{pmatrix}
0 & 1 & 0 & \sqrt{-1}r_1^3 & r_1\\
0 & 0 & 1 & r_1 & \sqrt{-1} \\
1 & 0 & 0 & 0 & r_2 \\
\end{pmatrix}
,
\end{equation*}
where $r_1$ is a positive number such that $1, r_1, r_1^2, r_1^3$ are
linearly independent over ${\mathbb Q}$, and $r_2 \in {\mathbb R}
\setminus {\mathbb Q}$. We denote by $\Gamma $ a discrete subgroup of
${\mathbb C}^3$ generated by column vectors of $P$. Then
$X = {\mathbb C}^3/\Gamma $ is a toroidal group, for
$r_2 \in {\mathbb R} \setminus {\mathbb Q}$. It is obvious that $X$
is a quasi-abelian variety of kind 0. \par
Let ${\mathbb A} = {\mathbb C}^2/\Lambda _0$ be an abelian variety
with the following period matrix
\begin{equation*}
\begin{pmatrix}
1 & 0 & \sqrt{-1}r_1^3 & r_1\\
0 & 1 & r_1 & \sqrt{-1}\\
\end{pmatrix}
.
\end{equation*}
We can represent $X$ as a principal ${\mathbb C}^{*}$-bundle
$\rho : X \longrightarrow {\mathbb A}$ over ${\mathbb A}$.
We denote by ${\mathbb R}_{\Gamma }^{5}$ the real linear subspace
spanned by $\Gamma $. Then ${\mathbb C}_{\Gamma }^{2} :=
{\mathbb R}_{\Gamma }^{5} \cap \sqrt{-1} {\mathbb R}_{\Gamma }^{5}$
is the maximal complex linear subspace contained in ${\mathbb R}_{\Gamma }^{5}$.
Let $(z_1,z_2,z_3)$ be toroidal coordinates of ${\mathbb C}^3$.
The principal ${\mathbb C}^{*}$-bundle
$\rho : X \longrightarrow {\mathbb A}$ is given by a projection
$\mu : {\mathbb C}^3 \longrightarrow {\mathbb C}_{\Gamma }^{2}, \, 
(z_1,z_2,z_3) \longmapsto (z_1,z_2)$.
Assume that $X$ contains a 1-dimensional toroidal subgroup $Y_0$.
Since a 1-dimensional toroidal group is a complex torus, 
$Y_0$ is a complex torus. Then $Y_0$ is contained in the maximal
compact subgroup ${\mathbb R}_{\Gamma }^{5}/\Gamma $.
Therefore, the connected component of $\pi ^{-1}(Y_0)$ containing $0$
is in ${\mathbb C}_{\Gamma }^{2}$, where 
$\pi : {\mathbb C}^3 \longrightarrow X$ is the projection.
Hence, $\rho (Y_0) = \rho (\pi (\pi ^{-1}(Y_0)))$ is a
1-dimensional subtorus of ${\mathbb A}$. However, ${\mathbb A}$ is
simple by Lemma 2. This is a contradiction. Therefore,
$X$ does not contain a 1-dimensional toroidal subgroup.\par
Next we consider 2-dimensional toroidal subgroups. We assume that
there exist a toroidal group ${\mathbb C}^2/\Lambda $ and a holomorphic
immersion $\varphi : {\mathbb C}^2/\Lambda \longrightarrow X$ which
is a homomorphism. Let $\Phi : {\mathbb C}^2 \longrightarrow {\mathbb C}^3$
be the linear extension of $\varphi $. We write
${\mathbb C}^3 = {\mathbb C}_{\Gamma }^2 \oplus V \oplus \sqrt{-1}V$
and ${\mathbb R}_{\Gamma }^5 = {\mathbb C}_{\Gamma }^2 \oplus V$ as usually,
where $V$ is a real linear subspace.
We set $W := \Phi ({\mathbb C}^2) \cap {\mathbb C}_{\Gamma }^2$.
When $\dim _{{\mathbb C}}W = 1$, 
$\Phi ({\mathbb C}^2) = W \oplus V \oplus \sqrt{-1}V$. We have
$\Phi (\Lambda ) \subset \Phi ({\mathbb C}^2 )\cap \Gamma = (W \oplus V)
\cap \Gamma $. Let $\Gamma _0 := {\mathbb Z} {\bf e}_3$, where
${\bf e}_3 = {}^t(0,0,1)$. Then there exists a discrete subgroup
$\Gamma _1$ such that $\Phi ({\mathbb C}^2) \cap \Gamma = \Gamma _0
\oplus \Gamma _1$. If $\Gamma _1 \not= \{ 0 \}$, then there exists
$\gamma ^{(0)} \in \Gamma _1$ such that
\begin{equation*}
\| \gamma ^{(0)} \| = \min \{ \| \gamma \| ; \gamma \in \Gamma _1,
\gamma \not= 0 \}.
\end{equation*}
We note $\Phi ({\mathbb C}^2) = {\mathbb C}{\bf e}_3 \oplus {\mathbb C}
\gamma ^{(0)}$. Using the same argument as in the proof of Lemma 2
for $L = {\mathbb C}\gamma ^{(0)}$ and $\Gamma _1$, we see
${\rm rank}(\Phi ({\mathbb C}^2) \cap \Gamma ) = 2$. 
Then $\varphi ({\mathbb C}^2/\Lambda ) = \Phi ({\mathbb C}^2)/(
\Phi ({\mathbb C}^2) \cap \Gamma )$ is not a toroidal group.\par
If $\dim _{{\mathbb C}}W = 2$, then
$\Phi ({\mathbb C}^2) = {\mathbb C}_{\Gamma }^2$. In this case,
${\mathbb C}_{\Gamma }^2 \cap \Gamma $ has a period matrix
\begin{equation*}
P'=
\begin{pmatrix}
1 & 0 & \sqrt{-1}r_1^3\\
0 & 1 & r_1\\
\end{pmatrix}.
\end{equation*}
Since $r_1 \notin {\mathbb Q}$, $Y := \varphi ({\mathbb C}^2/\Lambda ) =
{\mathbb C}_{\Gamma }^2/({\mathbb C}_{\Gamma }^2 \cap \Gamma )$ is a
toroidal subgroup of $X$. We have $\pi ^{-1}(Y) = {\mathbb C}_{\Gamma }^2 + \Gamma $.
By the density condition of toroidal groups, the closure
$\overline{{\mathbb C}_{\Gamma }^2 + \Gamma }$ of ${\mathbb C}_{\Gamma }^2 + \Gamma $
is equal to ${\mathbb R}_{\Gamma }^5$. Therefore, $X$ is the smallest closed
toroidal subgroup which contains $Y$. Hence, $X$ is geometrically simple,
but contains a 2-dimensional toroidal subgroup $Y$.

\section{Decomposition}
The following theorem is a generalization of Proposition 4.8
in \cite{takayama}.

\begin{theorem}
Let $X$ be an $n$-dimensional quasi-abelian variety. If $X$ contains a
closed quasi-abelian subvariety $X_1$, then there exists a closed
quasi-abelian subvariety $X_2$ of $X$ such that the natural
homomorphism $\varphi : X_1 \times X_2 \longrightarrow X$ is an
isogeny.
\end{theorem}

\begin{proof}
We may write $X = V/\Gamma $, where $V$ is an $n$-dimensional complex
linear space and $\Gamma $ is a discrete subgroup of $V$.
Let ${\rm rank}\; \Gamma = n+m$. We have a real linear subspace $W$
of dimension $n-m$ such that
\begin{equation*}
{\mathbb R}_{\Gamma }^{n+m} = {\mathbb C}_{\Gamma }^m \oplus W\quad
\text{and}\quad V = {\mathbb C}_{\Gamma }^m \oplus W \oplus \sqrt{-1}W.
\end{equation*}
There exists a complex linear subspace $V_1$ of $V$ such that
$X_1 = V_1/\Gamma _1$, where $\Gamma _1 := V_1 \cap \Gamma $. Let
$\dim _{{\mathbb C}}V_1 = n_1$ and ${\rm rank}\; \Gamma _1 = n_1 + m_1$.
We can take a real linear subspace $W_1$ of dimension $n_1 - m_1$ such that
\begin{equation*}
{\mathbb R}_{\Gamma _1}^{n_1+m_1} = {\mathbb C}_{\Gamma _1}^{m_1} \oplus W_1\quad
\text{and}\quad V_1 = {\mathbb C}_{\Gamma _1}^{m_1} \oplus W_1 \oplus \sqrt{-1}W_1.
\end{equation*}
We note that ${\mathbb C}_{\Gamma _1}^{m_1}$ is a complex linear
subspace of ${\mathbb C}_{\Gamma }^m$ and
${\mathbb C}_{\Gamma _1}^{m_1} \oplus W_1 \subset {\mathbb R}_{\Gamma }^{n+m}$.
\par
We set $k_1 := n_1 + m_1$. We take generators $\gamma _1, \dots , \gamma _{n+m}$
of $\Gamma $ such that $\gamma _1, \dots ,\gamma _{k_1}$ are generators of
$\Gamma _1$.
Since $X$ is a quasi-abelian variety, there exists an ample Riemann
form ${\mathcal H}$ for $X$. We denote by ${\mathcal A}$ the imaginary
part of ${\mathcal H}$. We set
\begin{equation*}
A_1 := 
\begin{pmatrix}
{\mathcal A}(\gamma _1, \gamma _1) & \cdots &{\mathcal A}(\gamma _1, \gamma _{k_1})
& {\mathcal A}(\gamma _1, \gamma _{k_1 +1}) & \cdots &
{\mathcal A}(\gamma _1, \gamma _{n+m})\\
\vdots &    & \vdots & \vdots &      & \vdots  \\
{\mathcal A}(\gamma _{k_1}, \gamma _1) & \cdots &{\mathcal A}(\gamma _{k_1}, \gamma _{k_1})
& {\mathcal A}(\gamma _{k_1}, \gamma _{k_1 +1}) & \cdots &
{\mathcal A}(\gamma _{k_1}, \gamma _{n+m})\\
\end{pmatrix}.
\end{equation*}
Then $A_1$ is a $(k_1, n+m)$-matrix with integral entries. 
Let $r_1 := {\rm rank}\; A_1$.
Since ${\mathcal H}$ is positive definite on ${\mathbb C}_{\Gamma _1}^
{m_1}$ and ${\mathcal A}$ is an alternating form, we have
$r_1 = 2(m_1 + k)$ with $0 \leqq 2k \leqq n_1 - m_1$. We consider
an equation $A_1 {\bf x} = {\bf 0}$ $({\bf x} \in {\mathbb R}^{n+m})$.
Let $S_1$ be the space of solutions of this equation. Then we have
$\dim _{{\mathbb R}}S_1 = n + m -r_1$. Since $r_1 \leqq k_1$, we can
take ${\bf x}^{(1)}, \dots , {\bf x}^{(n+m-k_1)} \in S_1 \cap 
{\mathbb Z}^{n+m}$ which are linearly independent over ${\mathbb R}$
such that $\det \left( x^{(j)}_{k_1 + i}\right) _{i,j = 1, \dots ,
n+m-k_1} \not= 0$, where
${\bf x}^{(j)} = {}^t (x_1^{(j)}, \dots , x_{k_1}^{(j)},
x_{k_1 + 1}^{(j)}, \dots , x_{n+m}^{(j)})$. We set
\begin{equation*}
\lambda ^{(j)} := \sum _{i=1}^{n+m} x_i^{(j)}\gamma _i \in
\Gamma \setminus \Gamma _1
\end{equation*}
for $j=1, \dots , n+m-k_1$. Then $\lambda ^{(1)},\dots ,
\lambda ^{(n+m-k_1)}$ are linearly independent over ${\mathbb R}$.
We define a subgroup $\Lambda $ of $\Gamma $ with
${\rm rank}\; \Lambda = n+m-k_1$ by
\begin{equation*}
\Lambda := \bigoplus _{j=1}^{n+m-k_1} {\mathbb Z}\lambda ^{(j)}.
\end{equation*}
Let $\ell := n + m - k_1$. Since $\gamma _1, \dots , \gamma _{k_1},
\lambda ^{(1)}, \dots , \lambda ^{(n+m-k_1)}$
are linearly independent over ${\mathbb R}$, we have
${\mathbb R}_{\Lambda }^{\ell } \cap {\mathbb R}_{\Gamma _1}^{k_1}
= \{ 0 \}$. Then we obtain ${\mathbb R}_{\Gamma }^{n+m} =
{\mathbb R}_{\Lambda }^{\ell } \oplus {\mathbb R}_{\Gamma _1}^{k_1}$.
Therefore we have
\begin{equation*}
\overline{{\mathbb R}_{\Lambda }^{\ell } + \Gamma } =
\overline{{\mathbb R}_{\Lambda }^{\ell } \oplus \Gamma _1} =
{\mathbb R}_{\Lambda }^{\ell } \oplus \Gamma _1 =
{\mathbb R}_{\Lambda }^{\ell } + \Gamma .
\end{equation*}
There exists a complex linear subspace $E$ such that
${\mathbb C}_{\Gamma }^m = {\mathbb C}_{\Gamma _1}^{m_1} \oplus E$.
Then we have
\begin{equation*}
W_1 = (W_1 \cap E) \oplus (W_1 \cap W).
\end{equation*}
Furthermore, there exist a real linear subspace $F$ and a complex
linear subspace $E_0$ such that
$W = (W_1 \cap W) \oplus F$ and
\begin{equation*}
E = E_0 \oplus (W_1 \cap E) \oplus \sqrt{-1}(W_1 \cap E).
\end{equation*}
We have ${\mathbb R}_{\Gamma _1}^{k_1} = {\mathbb C}_{\Gamma _1}^{m_1}
\oplus W_1 $ and
\begin{equation*}
{\mathbb R}_{\Gamma }^{n+m} = {\mathbb C}_{\Gamma _1}^{m_1}
\oplus E_0 \oplus (W_1 \cap E) \oplus \sqrt{-1}(W_1 \cap E)
\oplus (W_1 \cap W) \oplus F.
\end{equation*}
On the other hand, we have
\begin{equation*}
\begin{split}
{\mathbb R}_{\Gamma }^{n+m} & =
{\mathbb R}_{\Lambda }^{\ell } \oplus {\mathbb R}_{\Gamma _1}^{k_1}\\
& = {\mathbb R}_{\Lambda }^{\ell } \oplus {\mathbb C}_{\Gamma _1}^{m_1}
\oplus (W_1 \cap E) \oplus (W_1 \cap W).
\end{split}
\end{equation*}
Then we obtain
\begin{equation*}
{\mathbb R}_{\Lambda }^{\ell } = E_0 \oplus \sqrt{-1}(W_1 \cap E) 
\oplus F.
\end{equation*}
If we define
\begin{equation*}
V_2 := E_0 \oplus (W_1 \cap E) \oplus \sqrt{-1}(W_1 \cap E)
\oplus F \oplus \sqrt{-1}F,
\end{equation*}
then $V_2$ is a complex linear subspace of $V$. We note that
$\Lambda $ is a discrete subgroup of $V_2$ with $\Lambda \subset \Gamma $.
We set $\Gamma _2 := V_2 \cap \Gamma $ and $X_2 := V_2/\Gamma _2$.
Since $\Lambda \subset \Gamma _2$, we have 
$\ell = {\rm rank}\; \Lambda \leqq {\rm rank}\; \Gamma _2$.
From $\Gamma _1 + \Lambda \subset \Gamma _1 + \Gamma _2 \subset \Gamma $
and ${\rm rank}(\Gamma _1 + \Lambda ) = n + m$ it follows that
\begin{equation*}
{\rm rank}(\Gamma _1 + \Lambda ) = {\rm rank}(\Gamma _1 + \Gamma _2)
= {\rm rank}\; \Gamma = n + m.
\end{equation*}
Then we obtain 
${\rm rank}\; \Gamma _2 = {\rm rank}\; \Lambda = \ell $
and ${\mathbb R}_{\Lambda }^{\ell } = {\mathbb R}_{\Gamma _2}^{\ell }$.
Therefore we have
\begin{equation*}
\overline{{\mathbb R}_{\Gamma _2}^{\ell } + \Gamma } =
\overline{{\mathbb R}_{\Lambda }^{\ell } + \Gamma } =
{\mathbb R}_{\Lambda }^{\ell } + \Gamma =
{\mathbb R}_{\Gamma _2 }^{\ell } + \Gamma .
\end{equation*}
This means that $X_2$ is a closed complex Lie subgroup of $X$.
Since $V_1 \oplus V_2 = V$, the natural homomorphism
$\varphi : X_1 \times X_2 \longrightarrow X$ is surjective.
Then the linear extension $\Phi : V_1 \oplus V_2 \longrightarrow V$
of $\varphi $ is bijective. Therefore $\varphi $ is an isogeny, for
$\Gamma /(\Gamma _1 + \Gamma _2)$ is a finite group.
Hence $X_2$ is a closed quasi-abelian subvariety of $X$.
\end{proof}

The following corollary is immediate from Theorem 3.

\begin{corollary}
Let $X$ be a quasi-abelian variety. Then, there exist a finite number
of geometrically simple quasi-abelian subvarieties
$X_1, \dots , X_k$ of $X$ such that $X$ and 
$X_1 \times \cdots \times X_k$ are isogenous.
\end{corollary}

\noindent
{\bf Remark.}
Let ${\mathbb A}$ be an abelian variety. Then ${\mathbb A}$ is
isogenous to a product ${\mathbb A}_1\times \cdots \times
{\mathbb A}_N$ of simple abelian varieties ${\mathbb A}_1, \dots 
, {\mathbb A}_N$. Furthermore,
this decomposition ${\mathbb A}_1\times \cdots \times
{\mathbb A}_N$ is unique up to isogeny. 
A proof of the uniqueness is based on the following fact.
\\
Let ${\mathbb A}$ and ${\mathbb A}'$ be simple abelian varieties.
If ${\mathbb A}$ and ${\mathbb A}'$ are not isogenous, then
${\rm Hom}({\mathbb A},{\mathbb A}') = \{ 0 \}$.\par
Unfortunately, it does not hold in the case of quasi-abelian varieties.
The example in Section 2 gives a counterexample.
Let $X_1$ and $X_2$ be the 2-dimensional toroidal subgroup $Y$ and
the 3-dimensional quasi-abelian variety $X$ in the example in
Section 2 respectively. Since $Y$ is non-compact, it is a simple 
quasi-abelian variety, hence geometrically simple. We have a non-zero 
homomorphism $\varphi : X_1 \longrightarrow X_2$ as shown in Section 2.
Then ${\rm Hom}(X_1,X_2) \not= \{ 0 \} $.\par
Therefore, we cannot apply the above proof of the uniqueness for
abelian varieties to quasi-abelian varieties.
We do not know whether the uniqueness holds in
the case of quasi-abelian varieties.

\section{Endomorphisms}
We extend Lemma 3.3 in \cite{abe} to geometrically simple toroidal groups
in this section.

\begin{proposition}
If $X = {\mathbb C}^n/\Gamma $ is a geometrically simple toroidal
group, then ${\rm End}_{{\mathbb Q}}(X)$ is a division algebra
over ${\mathbb Q}$.
\end{proposition}

\begin{proof}
Take any $\varphi \in {\rm End}(X)$ with $\varphi \not= 0$.
Let $\Phi : {\mathbb C}^n \longrightarrow {\mathbb C}^n$ be
the linear extension of $\varphi $. It suffices to show that
$\Phi $ is injective. We denote by $K$ the kernel of $\Phi$.
Assume $K \not= \{ 0\}$. Since $\Phi (\Gamma ) \subset \Gamma $,
$\Phi |_{{\mathbb R}_{\Gamma }^{n+m}} : {\mathbb R}_{\Gamma }^{n+m}
\longrightarrow {\mathbb R}_{\Gamma }^{n+m}$ is a real linear mapping,
where ${\rm rank}\; \Gamma = n+m$. We set $W := \Phi ({\mathbb C}^n)$.
Then we have
\begin{equation*}
{\mathbb R}_{\Gamma }^{n+m} = (K \cap {\mathbb R}_{\Gamma }^{n+m})
\oplus (W \cap {\mathbb R}_{\Gamma }^{n+m}).
\end{equation*}
It is easy to see that ${\rm rank}(K \cap \Gamma + \Phi (\Gamma ))
= {\rm rank}\; \Gamma $ and $(K \cap \Gamma ) \cap \Phi (\Gamma ) =
\{ 0 \}$.
Then $\Gamma /(K\cap \Gamma + \Phi (\Gamma ))$ is a finite group.
Therefore, the natural homomorphism
$\mu : K/(K \cap \Gamma ) \oplus W/\Phi (\Gamma ) \longrightarrow X$
is an isogeny. Hence we obtain an isogeny
$\nu : X \longrightarrow K/(K \cap \Gamma ) \oplus W/\Phi (\Gamma )$.
Both $K/(K \cap \Gamma )$ and $W/\Phi (\Gamma )$ are toroidal groups.
Then $\nu ^{-1}(K/(K \cap \Gamma ))$ is a closed toroidal subgroup
of $X$. It contradicts the assumption. Hence $K = \{ 0 \}$.
This completes the proof.
\end{proof}

\vspace{0.5cm}
\noindent
{\it Acknowledgement.}
The author would like to thank the referees for their careful
reading of the manuscript and valuable comments.

\flushleft{
Graduate School of Science and Engineering for Research\\
University of Toyama\\
Toyama 930-8555, Japan\\}

\noindent
e-mail: abe@sci.u-toyama.ac.jp\\

\end{document}